\documentclass[english]{amsart}
\usepackage[T1]{fontenc}
\usepackage{amsmath}
\usepackage{amsthm}
\usepackage[style=alphabetic, maxalphanames=4, maxbibnames=99]{biblatex}
\usepackage{csquotes}
\usepackage{hyperref}
\usepackage{amssymb}
\usepackage{amsfonts}
\usepackage{enumitem}
\usepackage{biblatex}
\usepackage{float}

\usepackage{graphicx}
\graphicspath{ {./} }
\makeatletter
\theoremstyle{plain}

\newtheorem{thm}{\protect\theoremname}[section]
\theoremstyle{plain}

\newtheorem{prop}[thm]{Proposition}
\newtheorem{cor}[thm]{Corollary}

\theoremstyle{definition}
\newtheorem{defn}[thm]{\protect\definitionname}
\newtheorem{exs}[thm]{Examples}

\newtheorem{con}[thm]{Construction}
\newtheorem{que}[thm]{Question}

\DeclareMathOperator{\Stab}{Stab} 
\DeclareMathOperator{\VC}{VC} 

\usepackage{tikz-cd}
\makeatother
\usepackage{babel}
\providecommand{\definitionname}{Definition}

\providecommand{\lemmaname}{Lemma}
\providecommand{\questionname}{Question}
\providecommand{\theoremname}{Theorem}
\usepackage{graphicx}

\title{Set Systems with Covering Properties and Low VC-Dimension}
\author{George Peterzil\, }
\address{Einstein Institute of Mathematics \\
The Hebrew University of Jerusalem \\ 91904 Jerusalem \\ Israel \vspace{7pt}}
\email{george.peterzil@mail.huji.ac.il}
\author{\,Johanna Steinmeyer}
\email{johanna.steinmeyer@mail.huji.ac.il}
\date{January 20, 2024}

\begin{document}

\maketitle
\vspace{-0.5cm}
\begin{abstract}
  Given natural numbers $k\leq s\leq n$, we ask: what is the minimal VC-dimension of a family $\mathcal{F}$ of $s$-subsets of $[n]$ that covers all $k$-subsets of $[n]$? We first show that for sufficiently large $n$ this number is always $k$, and construct families which give a lower bound for the actual growth of this stabilization point. 

\end{abstract}

\vspace{1cm}

Consider a family of sets $\mathcal{F}$.
We say that a set $A$ is \textit{shattered by $\mathcal{F}$} if $$2^A=\left\{A\cap S\mid S\in\mathcal{F}\right\}.$$
Define the \textit{VC-dimension} of~$\mathcal{F}$ to be the size of the largest finite set shattered by~$\mathcal{F}$, or $\infty$ if there is no such maximum. 

 Originiating from machine learning in~\cite{VC}, the notion of VC-dimension has since been further studied also by combinatorialists (e.g.~\cite{FKKP} and~\cite{AMY}, for an overview see~\cite{Dud}) as well as logicians (for a survey see~\cite{Simon}). 
 
 Consider the family of initial segments of the natural numbers. It demonstrates that one can cover all finite subsets of $\mathbb{N}$ with a family of finite sets of VC-dimension~$1$. Bays, Ben-Neria, Kaplan and Simon prove the following theorem, which can be seen as an uncountable analogue of the above fact:
\begin{thm}[\cite{Low}, Proposition 3.3 and Theorem 3.8]
    Every family of finite sets which covers every finite subset of the first uncountable cardinal has VC-dimension at least $2$, and there is such a family with VC-dimension exactly $2$.
\end{thm}
This result is motivated by model theory, namely by the existence of so-called honest definitions for dependent local types (see \cite{Dep}, Definition 1.6). This means that sets which are definable with a parameter \emph{external} to a first-order structure by a formula of finite VC-dimension have their finite subsets covered by a definable family (with an \emph{internal} parameter). It is therefore natural to ask for effective results: suppose $k<s<n$ are natural numbers. What is the minimal VC-dimension $D(k,s,n)$ of a family of subsets of~$[n]$, each of size~$s$, which covers every $k$-element subset of $[n]$? 

Finitary questions regarding VC-classes appear across the literature. A celebrated one of these is the following, proved independently by Sauer and Shelah.
\begin{thm}[\cite{Sau},\cite{She}]
\label{SauShe}
    Given $\mathcal{F}\subseteq 2^{[n]}$, if $|\mathcal{F}|>\sum_{i=0}^{k-1}\binom{n}{i}$ then the VC-dimension of $\mathcal{F}$ is at least $k$.
\end{thm}
We can see our question above as an analogue of the Sauer-Shelah lemma, replacing the assumption on the size of the family by the covering assumptions. The following is our main result.
\begin{thm}[\ref{main}]
        For all $k\leq s$ and $n\in\mathbb{N}$ with $k^2\binom{s}{k}+k\leq n$ we have $D(k,s,n)=k$.
\end{thm}
In addition to this result, we construct $k$-covering families with VC-dimension smaller than $k$ on smaller sets to demonstrate the requirement for large $n$ in the above theorem. These results utilize particular examples of families as well as canonical constructions of new families with prescribed properties from existing ones. 

Answering analogous questions for infinite cardinals is part of work in progress of Omer Ben-Neria, Itay Kaplan and the first author.
\section*{Acknowledgements}
This work is part of G.P.'s Master's thesis, done in the Hebrew University of Jerusalem. He would like to thank his advisors Omer Ben-Neria and Itay Kaplan for introducing him to this flavor of questions. G.P. is supported by the Israel Science Foundation (Grant 1832/19). J.S. is supported by the Schaerf fund of the Einstein Institute of Mathematics.
\section{Definitions and Examples}\
Given a natural number $n$, write $[n]=\{1,...,n\}$. For a set $A$ and $s\in\mathbb{N}$, we write $\binom{A}{s}=\big\{A_0\subseteq A: |A_0|=s\big\}$.
\begin{defn}
    Let $\mathcal{F}$ be a family of sets.
    \begin{itemize}
        \item Say that a set $A$ is \textit{shattered by $\mathcal{F}$} if $2^A=\left\{A\cap S\mid S\in\mathcal{F}\right\}$.
        \item Define the \textit{VC-dimension} of $\mathcal{F}$ to be the size of the largest finite set shattered by $\mathcal{F}$, or $\infty$ if there is no such maximum. 
    \end{itemize}
\end{defn}
\begin{exs}
    \begin{itemize}
        \item \label{all} Given natural numbers $s\leq n$, the family $\binom{[n]}{s}$ has VC-dimension $\min\{s,n-s\}$; the set $[\min\{s,n-s\}]\subseteq[n]$ is shattered, as any $A\subseteq[\min\{s,n-s\}]$ of size $k$ has $A=[s]\cap (A\cup\{s+1,\dots,2s-k\})$, and any larger set does not have the empty set as the intersection of it with a member of the family by the pigeonhole principle.
        \item The family $\mathcal{F}$ of proper initial segments in a linearly ordered set $X$ with at least two distinct elements $x<y$ has VC-dimension $1$. Indeed, it is at least $1$ since there is $I\in\mathcal{F}$ containing $y$ and there is $J\in\mathcal{F}$ which contains $x$ and not $y$, hence $\{y\}$ is shattered. However, the set $\{x,y\}$ is not shattered -- every $I\in\mathcal{F}$ which contains $y$ contains $x$ as well.
    \end{itemize}
\end{exs}
\begin{defn}
    \begin{itemize}
    \item Let $X$ be a set, $\mathcal{F}\subseteq 2^X$, $k\in\mathbb N$ with $k\leq |X|$. Say that $\mathcal{F}$ has the $\textit{k-covering property}$ if every $A\in\binom{X}{k}$ is contained in a member of $\mathcal{F}$. 
    \item Given $k\leq s\leq n$, write $D(k,s,n)$ for the smallest VC-dimension of a family $\mathcal{F}\subseteq\binom{[n]}{s}$ which has the $k$-covering property. 
    \end{itemize}
\end{defn}
Below are a few simple instances of calculations of this function.
\begin{exs}
    \begin{itemize}
        \item For all $k\leq n$ we have $D(k,n,n)=0$, as witnessed by the trivial family $\{[n]\}$.
        \item For all $k\leq s$ we have $D(k,s,n)\leq\min\{s,n-s\}$ by \ref{all}.
        \item \label{k,k} Since the only $\mathcal{F}\subseteq \binom{[n]}{k}$ with the $k$-covering property is $\binom{[n]}{k}$ itself, we have $D(k,k,n)=\min\{k,n-k\}$ by $\ref{all}$.
    \end{itemize}
\end{exs}
\section{General Bounds and Constructions}
Very few families with covering properties can have VC-dimension $1$. The following is a simple instance of this.

\begin{prop}
    For $2\leq k\leq s$ such that $s<2k$ and $s+k<n$ we have $1<D(k,s,n)$.
\end{prop}
\begin{proof}
    Take $S_1\in\mathcal{F}$. Using the $k$-covering property, pick $k$ points in $[n]\setminus S_1$ and some $S_2\in\mathcal{F}$ which covers them, so $|S_1\cap S_2|\leq s-k < k$. Take some $T\in\mathcal{F}$ which contains $S_1\cap S_2$ and some point outside of $S_1\cup S_2$. Since all members of $\mathcal{F}$ have size $s$, $T$ does not contain $S_1\setminus S_2$ nor $S_1\setminus S_2$, hence there is some $p_1\in S_1\setminus \big(S_2\cup T\big)$ and some $p_2\in S_2\setminus \big(S_1\cup T\big)$. Then $\big\{p_1,p_2\big\}$ is shattered by $S_1,S_2,T$ and any set which contains both $p_1$ and $p_2$.
\end{proof}

\begin{thm}
\label{bound}
    Fix $k\leq s$. Then for sufficiently large $n\in\mathbb{N}$ we have $k\leq D(k,s,n)$. In particular, one can take any $n\geq k^2\binom{s}{k}+k$. 
\end{thm}
\begin{proof}
    A short calculation shows that this is equivalent to $k\binom{n}{k-1}<\frac{\binom{n}{k}}{\binom{s}{k}}$.
    Now since $2k\leq n$, this gives us $\sum_{i=0}^{k-1}\binom{n}{i}<\frac{\binom{n}{k}}{\binom{s}{k}}$.
    Let $\mathcal{F}$ be a $k$-covering family of subsets of~$[n]$ of size $s$, then any $S\in\mathcal{F}$ can cover at most $\binom{s}{k}$ distinct subsets of~$[n]$ of size $k$. Since there are $\binom{n}{k}$ such subsets, the cardinality of $\mathcal{F}$ is at least $\frac{\binom{n}{k}}{\binom{s}{k}}$, hence by \ref{SauShe} we have $\VC(\mathcal{F})\geq k$.
\end{proof}

Our original approach to this bound was based on proving a finitary version of Proposition 3.8 of \cite{Low}, however it gave significantly worse estimates for~$n$ than the proof above.

The following two Propositions demonstrate the flexibility of pushing up~$s$ and~$n$, while keeping the VC-dimension and the $k$-covering property.
\begin{prop}
    Suppose $\mathcal{F}\subseteq\binom{[n]}{s}$ has VC-dimension $m$ and the $k$-covering property.
    \begin{enumerate}[ref={\theprop.\arabic*}]
        \item \label{cone} There is a family $\mathcal{F}_*\subseteq\binom{[n+1]}{s+1}$ of VC-dimension $m$ with the $k$-covering property (in fact, this family shatters the same sets as $\mathcal{F})$. In particular, for every $k\leq s\leq n$ we have $D(k,s+1,n+1)\leq D(k,s,n)$. 
        \item \label{box} For every natural number $\ell$, there is a family $\mathcal{F}\times \ell\subseteq\binom{[n\cdot \ell]}{s\cdot \ell}$ of VC-dimension $m$ with the $k$-covering property. In particular, for every $k\leq s\leq n$ we have $D(k,s\cdot \ell,n\cdot \ell)\leq D(k,s,n)$.
    \end{enumerate}
\end{prop}
\begin{proof}
    \begin{enumerate}
        \item Define $\mathcal{F}_*$ to be the family $\{A\cup\{n+1\}\mid A\in\mathcal{F}\}\subseteq\binom{[n+1]}{s+1}$. We will claim that it has the $k$-covering property. Indeed, take $A\subseteq [n+1]$ of size $k$. If it does not contain $n+1$, take some $B\in\mathcal{F}$ containing it and cover it with $B\cup\{n+1\}$. Otherwise, take some $x\in[n]\setminus A$ and some $B\in\mathcal{F}$ containing $A\cup\{x\}\setminus\{n+1\}$, then $A\cup\{n+1\}$ contains $A$. We will now claim that $\VC\left(\mathcal{F}_*\right)\leq \VC(\mathcal{F})$. Suppose $A\subseteq [n+1]$ is shattered by $\mathcal{F}_*$. If $A\subseteq[n]$ then necessarily $A$ is shattered by $\mathcal{F}$ since for every $S\in\mathcal{F}$ we have $S\cap A= S\cup\{n+1\}\cap A$. Otherwise we have $n+1\in A$, which implies that every $S\in\mathcal{F}_*$ has $n+1\in S\cap A$, a contradiction to $A$ being shattered by $\mathcal{F}_*$. We will now show that every set shattered by $\mathcal{F}$ is shattered by $\mathcal{F}_*$ as well. Indeed, suppose $A\subseteq[n+1]$ is shattered by $\mathcal{F}$, so necessarily $A\subseteq[n]$, hence for every $S\in\mathcal{F}$ we have $S\cup\{n+1\}\cap A=S\cap A$, so $A$ is shattered by $\mathcal{F}_*$ as well.
        \item Define $\mathcal{F}\times \ell\subseteq\binom{[n]\times[\ell]}{s\cdot \ell}$ as follows:
            \begin{align*}
            \mathcal{F}\times \ell=\{S\times[\ell]\mid S\in\mathcal{F}\}
            \end{align*}
         We will first claim that $\mathcal{F}\times \ell$ has the $k$-covering property. Fix $p_i=\left(v_i,x\right)\in[n]\times[\ell]$ for $1\leq i\leq n$. Take some $S\in\mathcal{F}$ which covers $\big\{v_i\mid 1\leq i\leq k\big\}$, so $S\times[\ell]$ covers $\big\{p_1,\dots,p_k\big\}$. We will now claim that $\VC(\mathcal{F}\times \ell)\leq m=\VC(\mathcal{F})$. Fix distinct points $p_i=\left(v_i,x_i\right)\in[n]\times[\ell]$ for $1\leq i\leq m+1$. If there are $1\leq i_1<i_2\leq m+1$ with $v_{i_1}=v_{i_2}$, then for every $S\in\mathcal{F}$ with $p_{i_1}\in S\times[\ell]$ we have $p_{i_2}\in S\times[\ell]$, hence $\mathcal{F}\times \ell$ does not shatter $\big\{p_i\mid 1\leq i\leq m+1\big\}$, so assume all $v_i$'s are distinct. By the definition of $v$, there is some $A\subseteq \big\{v_i\mid 1\leq i\leq m+1\big\}$ with $A\neq S\cap\big\{v_i\mid 1\leq i\leq m+1\big\}$ for all $S\in\mathcal{F}$, so $\big\{p_i\mid v_i\in A\big\}\neq (S\times [\ell])\cap\big\{p_i\mid 1\leq i\leq m+1\}$ for all $S\in\mathcal{F}$. To get equality in VC-dimension, note that for every $A\subseteq[n]$ and $S\in\mathcal{F}$, we have $(A\cap S)\times\{1\}=A\times\{1\}\cap S\times[\ell]$.
    \end{enumerate}
\end{proof}
\section{Specific Families}
The following is a family of examples that shows that for small enough $n$, we can expect low VC-dimension.
\begin{prop}
    For all $k,m\in\mathbb{N}$ we have the following bound: $$D(k, k^m,(k+1)^m)\leq m\log_2 (k+1)$$
\end{prop}
\begin{proof}
    Consider the abelian group $A=\big(\mathbb{Z}/(k+1)\big)^m$. Define the family
    $$\mathcal{F}=\left\{\prod_{i=1}^m X_i\mid X_i\in \binom{\mathbb{Z}/(k+1)}{k}\right\}\subseteq \binom{A}{k^m},$$
     where the members $\mathcal{F}$ are $m$-hypercubes of width $k$ in $A$. For every $1\leq i\leq m$ and $p\in A$ define $\pi_i(p)$ to be the $i$'th coordinate of $p~$. We will claim that $\mathcal{F}$ has the $k$-covering property. Indeed, given points $p_1,\dots,p_k\in A$, for every $1\leq i\leq m$, there is at least one $1\leq t_i\leq k+1$ such that $\pi_i(p_j)\neq t_i$ for all $1\leq j\leq k$. This means that $p_1,...,p_k$ are all contained in the following hypercube: 
    $$\prod_{i=1}^{m}\big(\big(\mathbb{Z}/(k+1)\big)\setminus\big\{t_i\big\}\big)$$
    
    This gives us the $k$-covering property. As for the VC-dimension, note that to shatter a set of size $t$, the family must be of cardinality at least $2^t$. Since $\mathcal{F}$ contains $|A|=(k+1)^m$ many elements, the result follows.
\end{proof}

	\begin{figure}[H]
		\centering
		\includegraphics[width=0.3\linewidth]{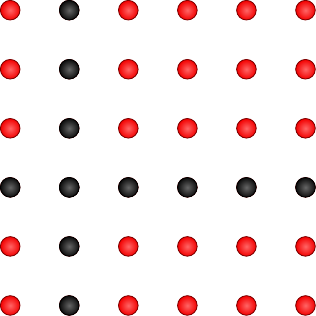}
		\caption{\small A member of $\mathcal{F}$ for $k=5$. }
		\label{fig:box}
	\end{figure}

The following demonstrates that we can achieve the bound in \ref{bound} for all $n$.
\begin{defn}
    Fix $s\leq n$ and a family of sets $\mathcal{F}\subseteq\binom{[n]}{s}$. Say that $\mathcal{F}$ has the \emph{unique face property} if for every $S\in\mathcal{F}$ there is $K\subsetneq S$ such that the only member of $\mathcal{F}$ containing $K$ is $S$.
\end{defn}
\begin{prop}
\label{ufp}
    If $\mathcal{F}\subseteq\binom{[n]}{s}$ has the unique face property, then it has VC-dimension less than $s$.
\end{prop}
\begin{proof}
    Fix $A\in\binom{[n]}{s}$. If $A\notin\mathcal{F}$ then $A$ itself is not of the form $A\cap S$ for any $S\in\mathcal{F}$. Otherwise, by the unique face property we have some $K\subseteq A$ such that the only member of $\mathcal{F}$ containing it is $A$. In particular, $K$ cannot be separated from $A\setminus K$ by $\mathcal{F}$.
\end{proof}
\begin{con}
    Given $m\in\mathbb{N}$ and $k=1$, consider the following family: \[\mathcal{F}_1=\{\{2t-1,2t\}\mid t\leq\frac{m}{2}\}\cup\{\{m-1,m\}\}\subseteq\binom{[m]}{2}\]
    Assuming we have defined $\mathcal{F}_k\subseteq\binom{[m+k-1]}{k}$, define the family $\mathcal{F}_{k+1}\subseteq\binom{[m+k]}{k+1}$ as follows:
    \begin{align*}
        \mathcal{F}_{k+1}=\bigcup_{i=1}^{m+k}\{S\cup\{i\}\mid S\in\mathcal{F}\cap 2^{[i-1]}\}
    \end{align*}
\end{con}

	\begin{figure}[H]
		\centering
		\includegraphics[width=0.5\linewidth]{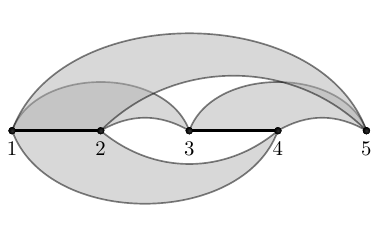}
		\caption{\small $\mathcal{F}_2$ when starting with $m=4$. }
		\label{fig:con}
	\end{figure}
	
\begin{prop}
\label{const}
    For all $m,k\in\mathbb{N}$, the family $\mathcal{F}_k\subseteq\binom{[n]}{k+1}$ (for $n=m+k-1$) satisfies the following properties:
    \begin{enumerate}
        \item It has the $k$-covering property.
        \item It has the unique face property.
        \item \label{3} For every ordered $\{t_1,\dots,t_{k+1}\}\in\mathcal{F}_k$ and $\hat{t}\in[n]$ with $t_k<\hat{t}<t_{k+1}$, we have $\{t_1,\dots,t_k,\hat{t}\}\in\mathcal{F}_k$ as well.
        \item If\, $2k<n$ then $\big\{n-k+1,\dots,n\big\}$ is shattered by $\mathcal{F}_k$.
    \end{enumerate}
\end{prop}
\begin{proof}
    We will prove the result by induction on $k$. For $k=1$ the first two are trivial, the third is void, and the fourth follows from $\{1,2\}\in\mathcal{F}_1$ and $\{n-1,n\}\in\mathcal{F}_1$. 
    
    Assume now that these three properties hold for $\mathcal{F}_k$. Pick an ordered $K=\big\{t_1,\dots,t_{k+1}\big\}\subseteq[n+1]^{k+1}$ and let $S\in\mathcal{F}_k$ cover $K\setminus\big\{t_{k+1}\big\}$. If $S$ covers $K$ then so does $S\cup\{n+1\}\in\mathcal{F}_{k+1}$. If $S$ does not cover $K$ and $\max{S}<t_{k+1}$ then $S\cup\big\{t_{k+1}\big\}$. Otherwise, there is some $r>0$ with $t_{k+1}+r=\max{S}$ (since $|S|=k+1)$. Applying~\ref{3} to $S$ and $\hat{t}=t_{k+1}$ gives us $S'\in\mathcal{F}_k$ with $K\subseteq S$, reducing to the first case. 
    
    For the unique face property, fix $S'\in\mathcal{F}_{k+1}$, so some $S\in\mathcal{F}_k$ has $S'=S\cup\{\max S'\}$. For $K_0$ a unique face of $S$, we will claim that $K=K_0\cup\{\max S'\}$ is a unique face of $S$. Suppose $T'\in\mathcal{F}_{k+1}$ has $K\subseteq T'$, so some $T\in\mathcal{F}_k$ has $T'=T\cup\{\max T'\}$. If $K\subseteq T$ then $K_0\subseteq T$, giving us $S=T$ and $\max T'=\max S'$, hence $S=T$ and $S'=T'$. Otherwise we have $\max T'\in K$, hence $\max T'=\max S'$ and so $S'=T'$ nonetheless. 
    
    Suppose $S'=\big\{s_1,...,s_{k+2}\big\}\in\mathcal{F}_{k+1}$ is ordered, so $S=S'\setminus\big\{s_{k+2}\big\}\in\mathcal{F}_k$. Fix $s_{k+1}<\hat{t}<s_{k+2}$. By definition of $\mathcal{F}_{k+1}$, since $\max S=s_{k+1}<\hat{t}$ we get $S\cup\{\hat{t}\}$ as well. 

    Assume that $2k+1<n$ and the claim holds up to $k$. By the induction hypothesis, $\{n-k+1,\dots,n\}$ is shattered by $\mathcal{F}_k$. To separate the empty set, note that $\{1,\dots,k+1\}\in\mathcal{F}_{k+1}$ and $k+1<n-k+1$. For $A\subseteq\{n-k+1,\dots,n+1\}$ nonempty, let $a\in A$ be maximal. By the induction hypothesis, there is some $S\in\mathcal{F}_k$ such that $S\cap \{n-k+1,\dots,n\}=A\setminus\{a\}$. In particular we have $S\subseteq 2^{[a-1]}$, so by definition we have $S\cup\{a\}\in\mathcal{F}_{k+1}$, hence $A=(S\cup\{a\})\cap\{n-k+1,\dots,n+1\}$.
\end{proof}
\begin{cor}
\label{upper}
    For all $k\leq s\leq n$ we have $D(k,s,n)\leq k$. 
\end{cor}
\begin{proof}
    For all $k\leq n$ have $D(k,k,n)\leq k$ by~\ref{k,k}. For $s=k+1$ we get the result by~\ref{const} and~\ref{ufp}. For general $s$, apply~\ref{cone} $s-k-1$ many times to the family $\mathcal{F}_k\subseteq\binom{[n-(s-k-1)]}{k+1}$.
\end{proof}
Joining \ref{upper} with \ref{bound} gives us our main theorem:
\begin{thm}
\label{main}
    For all $k\leq s$ and $n\geq k^2\binom{k}{s}+k$ we have $D(k,s,n)=k$.
\end{thm}
We end on a few natural questions arising from our analysis of the function $D(k,s,n)$. 
\begin{que}[Stabilization estimates]
    Theorem \ref{main} gives us that for fixed $k\leq s$, the function $D(k,s,-)$ stabilizes eventually on the value $k$, so given $k,s$, we can define $\Stab(k,s)$ to be the first $N$ such that for every $n\geq N$ we have $D(k,s,n)=k$. What is the actual growth rate of $\Stab(k,s)$? 
\end{que}
Another natural question is for the values that the function $D(k,s,-)$ attains before stabilizing.
\begin{que}[Surjectivity]
    Does the function $D(k,s,-)$ attain every natural number below $k$? 
\end{que}
If we had considered family of sets of size \emph{at most s} rather than \emph{exactly s}, the following question would have an immediate positive answer. However, in this case it seems quite tricky, and might have a negative answer due to divisibility issues.
\begin{que}[Monotonicity]
    When fixing two of the parameters and varying the third of $k,s,n$, is the corresponding function monotone?
\end{que}
\printbibliography

\end{document}